\documentclass[a4paper, 12pt]{amsart}

\usepackage{amssymb,amsmath}
\usepackage[dvips]{graphics}
\usepackage[dvips]{color}
\usepackage{graphicx}
\usepackage[english]{babel}
\usepackage[latin1]{inputenc}
\usepackage{comment}
\usepackage{todonotes}

\newtheorem{theorem}{Theorem}[section]
\newtheorem{lemma}[theorem]{Lemma}

\newtheorem{corollary}[theorem]{Corollary}

\theoremstyle{definition}
\newtheorem{definition}[theorem]{Definition}

\newtheorem{remark}[theorem]{Remark}

\numberwithin{equation}{section}

\title[Approximation by H\"older functions in Besov and T--L spaces]
{Approximation by H\"older functions in Besov and Triebel--Lizorkin spaces}

\author{Toni Heikkinen}
\author{Heli Tuominen}

\newcommand\rn{\mathbb R^n}
\newcommand\re{\mathbb R}
\newcommand\rv{\overline{\mathbb R}}
\newcommand\n{\mathbb N}
\newcommand\z{\mathbb Z}
\newcommand\D{\mathbb D}

\newcommand\ph{\varphi}
\newcommand\eps{\varepsilon}
\newcommand\M{\operatorname{\mathcal M}}

\newcommand\cH{\mathcal H}
\newcommand\dist{\operatorname{dist}}

\providecommand{\ch}[1]{\text{\raise 2pt \hbox{$\chi$}\kern-0.2pt}_{#1}}

\providecommand{\vint}[1]{\mathchoice
          {\mathop{\vrule width 5pt height 3 pt depth -2.5pt
                  \kern -9.5pt \kern 1pt\intop}\nolimits_{\kern -5pt{#1}}}%
          {\mathop{\vrule width 5pt height 3 pt depth -2.6pt
                  \kern -6pt \intop}\nolimits_{\kern -3pt{#1}}}%
          {\mathop{\vrule width 5pt height 3 pt depth -2.6pt
                  \kern -6pt \intop}\nolimits_{\kern -3pt{#1}}}%
          {\mathop{\vrule width 5pt height 3 pt depth -2.6pt
                  \kern -6pt \intop}\nolimits_{\kern -3pt{#1}}}}
\begin{document}

\begin{abstract}
In this paper, we show that Besov and Triebel--Lizorkin functions can be approximated by a H\"older continuous function both in the Lusin sense and in norm. The results are proven in metric measure spaces for Haj\l asz--Besov and Haj\l asz--Triebel--Lizorkin functions defined by a pointwise inequality. We also prove new inequalities for medians, including a Poincar\'e type inequality, which we use in the proof of the main result.
\end{abstract}


\date{\today}  
\maketitle

\section{Introduction}\label{intro}

\bigskip
By the classical Lusin theorem, each measurable function is continuous in a complement of a set of arbitrary small measure. 
For more regular functions, stronger versions of approximation results hold - the complement of the set where the function is not regular is smaller and is measured by a suitable capacity or a Hausdorff type content and the approximation can also be done in norm. 
This type of approximation by H\"older continuous functions for Sobolev functions was proven in \cite{Ma} by Mal\'y, who showed that each function $u\in W^{1,p}(\rn)$ coincides with a H\"older continuous function, that is close to $u$ in Sobolev norm, outside a set of small capacity.  
The result was strengthened by Bojarski, Haj\l asz and Strzelecki in \cite{BHS}, where they also discuss about the history of the  problem. For approximation results for Sobolev spaces, see also \cite{L}, \cite{MZ}, \cite{Sw1}, \cite{Sw2} and the references therein.

In the metric setting, approximation in Sobolev spaces $M^{1,p}(X)$, $p>1$, by H\"older continuous functions both in the Lusin sense, with the exceptional set measured using a Hausdorff content, and in norm, was studied by Haj\l asz and Kinnunen in \cite{HjKi}. The proof uses pointwise estimates, fractional sharp maximal functions and Whitney type smoothing. For the case $p=1$, see \cite{KiTu} and for fractional spaces, \cite{KP}.

\smallskip
In this paper, we study a similar approximation problem by H\"older continuous functions in Besov and Triebel--Lizorkin spaces in a metric measure space $X$ equipped with a doubling measure. We also assume that all spheres in $X$ are nonempty.
In the Euclidean case, Lusin type approximation in Besov and Triebel--Lizorkin spaces, and actually in more general spaces given by abstract definitions, has been done by Hedberg and Netrusov in \cite{HN}, see also \cite{Sto} (without a proof).

We prove the results for Haj\l asz--Besov spaces $N^s_{p,q}(X)$ and Haj\l asz--Triebel--Lizorkin spaces $M^s_{p,q}(X)$ which were recently introduced by Koskela, Yang and Zhou in \cite{KYZ} and studied for example in \cite{GKZ}, \cite{HeTu} and \cite{HIT}. 
This metric space approach is based on Haj\l asz type pointwise inequalities and it gives a simple way to define these spaces on a measurable subset of $\rn$ and on metric measure spaces. The definitions of spaces $N^s_{p,q}(X)$ and $M^s_{p,q}(X)$ as well as other definitions are given in Section \ref{preliminaries}.

In the Euclidean case, $N^s_{p,q}(\rn)=B^s_{p,q}(\rn)$ and $M^s_{p,q}(\rn)=F^s_{p,q}(\rn)$ for all $0<p<\infty$, $0<q\le\infty$, 
$0<s<1$, where $B^s_{p,q}(\rn)$ and $F^s_{p,q}(\rn)$ are Besov spaces and Triebel--Lizorkin spaces defined via an
$L^p$-modulus of smoothness, see \cite{GKZ}. 
Recall also that the Fourier analytic approach gives the same spaces when $p>n/(n+s)$ in the Besov case and when $p,q>n/(n+s)$ in the Triebel--Lizorkin case. Hence, for such parameters, our results hold also for the classical Besov and Triebel--Lizorkin spaces.

\smallskip
Our first main result tells that Besov functions can be approximated by H\"older continuous functions such that the approximating function coincides with the original function outside a set of small Hausdorff type content and the Besov norm of the difference is small.

\begin{theorem}\label{holder representative B}
Let $0<s<1$. Let $0<p,q<\infty$ and $0<\beta< s$ or $0<q\le p<\infty$ and $\beta=s$.
For each $u\in N^s_{p,q}(X)$ and $\eps>0$, there is an open set $\Omega$ and a function 
$v\in N^s_{p,q}(X)$ such that   
\begin{enumerate}
\item\label{holder 1}$u=v$ in $X\setminus \Omega$,  
\item\label{holder 2}$v$ is $\beta$-H\"older continuous on every bounded set of $X$,  
\item\label{holder 3}
$\|u-v\|_{ N^s_{p,q}(X)}<\eps$,
\item\label{holder 4}$\cH^{(s-\beta)p,q/p}_R(\Omega)<\eps$,
\end{enumerate}
where $R=2^6$.
\end{theorem}
Here $\cH^{(s-\beta)p,q/p}_R$ is the Netrusov--Hausdorff content of codimension $(s-\beta)p$, see Definition \ref{HN content}. 
Since the underlying measure is smaller than a constant times the Netrusov--Hausdorff content by Lemma \ref{measure and NH content}, the content estimate \eqref{holder 4} is stronger than a corresponding estimate for the measure.

In the case of Triebel--Lizorkin spaces, the exceptional set is measured by a Hausdorff content.
\begin{theorem}\label{holder representative TL}
Let $0<p<\infty$. Let $0<s<1$ and $0<q< \infty$ or $0<s\le 1$ and $q=\infty$. Let $0<\beta\le s$.
If $u\in M^s_{p,q}(X)$, then for any $\eps>0$, there is an open set $\Omega$ and a function 
$v\in M^s_{p,q}(X)$ such that   
\begin{enumerate}
\item $u=v$ in $X\setminus \Omega$,  
\item $v$ is $\beta$-H\"older continuous on every bounded set of $X$,  
\item 
$\|u-v\|_{ M^s_{p,q}(X)}<\eps$,
\item $\cH^{(s-\beta)p}_R(\Omega)<\eps$,
\end{enumerate}
where $R=2^6$.
\end{theorem}

In the case $q=\infty$, Haj\l asz--Triebel--Lizorkin space $M^s_{p,q}(X)$ coincides with the Haj\l asz space $M^{s,p}(X)$. 
Recall from \cite{H} that, for $p>1$, $M^{1,p}(\rn)=W^{1,p}(\rn)$, whereas for $n/(n+1)<p\le 1$,  $M^{1,p}(\rn)$ coincides with the Hardy--Sobolev space $H^{1,p}(\rn)$ by \cite[Thm 1]{KS}.
The following corollary of Theorem \ref{holder representative TL} extends the Sobolev space approximation results of \cite{HjKi}, \cite{KP}, and  \cite{KiTu} to the case $0<p<1$.

\begin{corollary}\label{holder representative H}
Let $0<s\le 1$, $0<p<\infty$, and let $0<\beta\le s$.
If $u\in M^{s,p}(X)$, then for any $\eps>0$, there is an open set $\Omega$ and a function 
$v\in M^{s,p}(X)$ such that   
\begin{enumerate}
\item $u=v$ in $X\setminus \Omega$,  
\item $v$ is $\beta$-H\"older continuous on every bounded set of $X$,  
\item $\|u-v\|_{ M^{s,p}(X)}<\eps$,
\item $\cH^{(s-\beta)p}_R(\Omega)<\eps$,
\end{enumerate}
where $R=2^6$.
\end{corollary}

In the proofs of approximation results, we use $\gamma$-medians $m^\gamma_u$ instead of integral averages. 
This enables us to handle also small parameters $0<p,q\le1$. 
Medians behave much like integral averages, but have an advantage that the function needs not be locally integrable.
We prove several new estimates relating a function and its (fractional) $s$-gradient in terms of medians.
One of these estimates is a version of a Poincar\'e inequality, which says that if $u$ is measurable and almost everywhere finite 
and $g$ is an $s$-gradient of $u$, then
\begin{equation}\label{median poincare 1}
\inf_{c\in\re}m_{|u-c|}^{\gamma}(B(x,r))
\le 2^{s+1}r^sm_{g}^{\gamma}(B(x,r))
\end{equation}
for every ball $B(x,r)$. 
We think that \eqref{median poincare 1} as well as Theorem \ref{mm}, which is a version of \eqref{median poincare 1}
for fractional $s$-gradients, are of an independent interest and not just tools in the proofs of our main results. 

The use of medians instead of integral averages also simplifies the proofs of certain estimates.
For example, the pointwise estimate
\begin{equation}\label{pw av}
|u(x)-u_{B(x,r)}|\le Cr^s\left(\M g^p(x)\right)^{1/p},
\end{equation}
where $Q/(Q+s)<p< 1$, requires a chaining argument and a Sobolev--Poincar\'e inequality, while
the corresponding estimate for medians,
\begin{equation}\label{pw med}
|u(x)-m_u^\gamma(B(x,r))|\le Cr^s\left(\M g^p(x)\right)^{1/p},
\end{equation}
is almost trivial and holds for all $p>0$.
The advantage of medians over integral averages becomes even more evident when one considers estimates like
\eqref{pw av} and \eqref{pw med} for fractional gradients, see Remark \ref{medians are better}.

\smallskip
The paper is organized as follows. 
In Section \ref{preliminaries}, we introduce the notation and the standard assumptions used in the paper and give the definitions of 
Haj\l asz--Besov and Haj\l asz--Triebel--Lizorkin spaces, $\gamma$-medians and Netrusov--Hausdorff content. 
In Section \ref{lemmas}, we present lemmas for contents and medians needed in the proof of the approximation result. 
Section \ref{proofs} is devoted to the proofs of the approximation results, Theorems \ref{holder representative B} and \ref{holder representative TL}. 
Finally, in the Appendix, we show that spaces $N^s_{p,q}(X)$ and $M^s_{p,q}(X)$ are complete. This is not proved in the earlier papers where these spaces are studied.

\section{Notation and preliminaries}\label{preliminaries}
In this paper,  $X=(X, d,\mu)$ is a metric measure space equipped with a metric $d$ and a Borel regular,
doubling outer measure $\mu$, for which the measure of every ball is positive and finite.
The {\em doubling} property means that there exists a constant $c_D>0$, called  {\em the doubling constant}, such that
\begin{equation}\label{doubling measure}
\mu(B(x,2r))\le c_D\mu(B(x,r))
\end{equation}
for every ball $B(x,r)=\{y\in X:d(y,x)<r\}$, where $x\in X$ and $r>0$.

We assume that $X$ has the \emph{nonempty spheres property}, that is, for every $x\in X$ and $r>0$, the set $\{y\in X: d(x,y)=r\}$ is nonempty. This property is needed to prove Poincar\'e type inequality \eqref{median poincare}, and it also enables us to simplify certain pointwise estimates, see Lemma \ref{pointwise lemma} and Remark \ref{medians are better}.

Note that the nonempty spheres property implies that all annuli have positive measure: Let $x\in X$, $r>0$, $0<\eps<r$ and let 
$A=B(x,r)\setminus B(x,r-\eps)$. By the assumption, there is $y$ such that $d(x,y)=r-\eps/2$. 
Now $B_y=B(y,\eps/2)\subset A$ and hence $\mu(A)\ge\mu(B_y)>0$.
 
By $\ch{E}$, we denote the characteristic function of a set $E\subset X$ and by $\rv$, the extended real numbers $[-\infty,\infty]$.
$L^0(X)$ is the set all measurable, almost everywhere finite functions $u\colon X\to\rv$.
In general, $C$ is positive constant whose value is not necessarily the same at each occurrence.
 
The integral average of a locally integrable function $u$ over a set  $A$ of finite and positive measure is denoted by
\[
u_A=\vint{A}u\,d\mu=\frac{1}{\mu(A)}\int_A u\,d\mu.
\]
The Hardy-Littlewood maximal function of $u$ is $\M u\colon X\to\rv$,
\[
\M u(x)=\sup_{r>0}\,\vint{B(x,r)}|u|\,d\mu.
\]

Our important tools are median values, which have been studied and used in different problems of analysis for example in  \cite{Fu}, \cite{GKZ}, \cite{HIT}, \cite{JPW}, \cite{JT}, \cite{Le}, \cite{LP}, \cite{PT}, \cite{St} and \cite{Zh}. 
In the theory of Besov and Triebel--Lizorkin spaces, they are extremely useful when $0<p\le1$ or $0<q\le 1$.
\begin{definition}
Let $0<\gamma\le 1/2$. The $\gamma$-{\em median} of a measurable function $u$ over a set $A$ of finite and positive  measure is
\[
m_u^\gamma(A)=\inf\big\{a\in\re: \mu(\{x\in A: u(x)>a\})< \gamma\mu(A)\big\},
\]
and the $\gamma$-{\em median maximal function} of $u$ is $M_\gamma u\colon X\to\rv$,
\[
M_\gamma u(x)=\sup_{r>0}\,m_{|u|}^\gamma(B(x,r)).
\]
\end{definition}
If $u\in L^0(A)$, then clearly $m_u^\gamma(A)$ is finite.
Note that the parameter $\gamma=1/2$ gives the standard median value of $u$ on $A$. It is denoted shortly by $m_u(A)$.

\subsection{Hajlasz--Besov and Hajlasz--Triebel--Lizorkin spaces}
There are several definitions for Besov and Triebel--Lizorkin spaces in metric measure spaces. 
We use the definitions given by pointwise inequalities in \cite{KYZ}.
The motivation for these definitions comes from the Haj\l asz--Sobolev spaces $M^{s,p}(X)$, defined for $s=1$ in \cite{H} and for fractional scales in \cite{Y}. 
We recall this definition below. For the other definitions for Besov and Triebel--Lizorkin spaces in the metric setting, see 
\cite{GKS}, \cite{GKZ}, \cite{HMY}, \cite{KYZ}, \cite{MY}, \cite{SYY}, \cite{YZ} and the references therein.

\begin{definition}
Let $s> 0$ and let $0<p<\infty$.
A nonnegative measurable function $g$ is an  {\em $s$-gradient} of a measurable function $u$, if 
\begin{equation}\label{eq: gradient}
|u(x)-u(y)|\le d(x,y)^s(g(x)+g(y))
\end{equation}
for all $x,y\in X\setminus E$, where $E$ is a set with $\mu(E)=0$. 
The {\em Haj\l asz space} $M^{s,p}(X)$ consists of measurable functions $u\in L^p(X)$ having an $s$-gradient in $L^p(X)$ and it is equipped with a norm (a quasinorm when $0<p<1$)
\[
\|u\|_{M^{s,p}(X)}=\|u\|_{L^p(X)}+\inf\|g\|_{L^p(X)},
\]
where the infimum is taken over all $s$-gradients of $u$.
\end{definition}

\begin{definition}
Let $0<s<\infty$.
A sequence of nonnegative measurable functions $(g_k)_{k\in\z}$ is a  {\em fractional $s$-gradient} of a measurable function
$u\colon X\to\rv$, if there exists a set $E\subset X$ with $\mu(E)=0$ such that
\[
|u(x)-u(y)|\le d(x,y)^s\big(g_k(x)+g_k(y)\big)
\]
for all $k\in\z$ and all $x,y\in X\setminus E$ satisfying $2^{-k-1}\le d(x,y)<2^{-k}$.
The collection of fractional $s$-gradients of $u$ is denoted by $\D^s(u)$.
\end{definition}

For $0<p,q\le\infty$ and a sequence $(f_k)_{k\in\z}$ of measurable functions, 
define
\[
\big\|(f_k)_{k\in\z}\big\|_{l^q(L^p(X))}
=\big\|\big(\|f_k\|_{L^p(X)}\big)_{k\in\z}\big\|_{l^q}
\]
and
\[
\big\|(f_k)_{k\in\z}\big\|_{L^p(X;\,l^q)}
=\big\|\|(f_k)_{k\in\z}\|_{l^q}\big\|_{L^p(X)},
\]
where
\[
\big\|(f_k)_{k\in\z}\big\|_{l^{q}}
=
\begin{cases}
\big(\sum_{k\in\z}|f_{k}|^{q}\big)^{1/q},&\quad\text{when }0<q<\infty, \\
\;\sup_{k\in\z}|f_{k}|,&\quad\text{when }q=\infty.
\end{cases}
\]

\begin{definition}
Let $0<s<\infty$ and $0<p,q\le\infty$.
The  {\em homogeneous Haj\l asz--Besov space} $\dot N_{p,q}^s(X)$ consists of measurable functions
$u\colon X\to\rv$, for which  the (semi)norm
\[
\|u\|_{\dot N_{p,q}^s(X)}=\inf_{(g_k)\in\D^s(u)}\|(g_k)\|_{l^q(L^p(X))}
\]
is finite, and the {\em (inhomogeneous) Haj\l asz--Besov space} $N_{p,q}^s(X)$ is $\dot N_{p,q}^s(X)\cap L^p(X)$
equipped with the norm
\[
\|u\|_{N_{p,q}^s(X)}=\|u\|_{L^p(X)}+\|u\|_{\dot N_{p,q}^s(X)}.
\]
Similarly, the {\em homogeneous Haj\l asz--Triebel--Lizorkin space} $\dot M_{p,q}^s(X)$ consists of measurable functions
$u\colon X\to\rv$, for which
\[
\|u\|_{\dot M_{p,q}^s(X)}
=\inf_{(g_k)\in\D^s(u)}\|(g_k)\|_{L^p(X;\,l^q)}
\]
is finite and the {\em Haj\l asz--Triebel--Lizorkin space} $M_{p,q}^s(X)$ is $\dot M_{p,q}^s(X)\cap L^p(X)$
equipped with the norm
\[
\|u\|_{M_{p,q}^s(X)}=\|u\|_{L^p(X)}+\|u\|_{\dot M_{p,q}^s(X)}.
\]
\end{definition}
When $0<p<1$, the (semi)norms defined above are actually quasi-(semi)norms, but for simplicity we call them, as well as other quasi-seminorms in this paper, just norms.

Note that, by the Aoki--Rolewicz Theorem, \cite{Ao}, \cite{Ro}, for each quasinorm $\|\cdot\|$, there is a comparable quasinorm 
$\||\cdot|\|$ and $0<r<1$ such that $\||u+v|\|^r\le \||u|\|^r+\||v|\|^r$ for each $u$ and $v$ in the quasinormed space.
Hence, if $0<p<1$ or $0<q<1$, the triangle inequality for the quasinorm does not hold but there are constants $0<r<1$ and $c>0$ such that  
\begin{equation}\label{quasi norm ie}
\Big\|\sum_{i=1}^\infty u_i\Big\|_{N^s_{p,q}(X)}^r\le c\sum_{i=1}^\infty\|u_i\|_{N^s_{p,q}(X)}^r
\end{equation}
whenever $u_i\in N^s_{p,q}(X)$. A corresponding result holds for Triebel--Lizorkin functions.

\subsection {On different definitions of Besov and Triebel--Lizorkin spaces}
The space $M^s_{p,q}(\rn)$ coincides with Triebel--Lizor\-kin space ${\bf F}^s_{p,q}(\rn)$, 
defined via the Fou\-rier analytic approach, when $s\in (0,1)$, $p\in(n/(n+s),\infty)$ and $q\in(n/(n+s),\infty]$, and
$M^1_{p,\infty}(\rn)=M^{1,p}(\rn)={\bf F}^{1}_{p,2}(\rn)$, when $p\in(n/(n+1),\infty)$.
Similarly, $N_{p,q}^s(\rn)$ coincides with Besov space ${\bf B}^s_{p,q}(\rn)$ for $s\in (0,1)$, $p\in(n/(n+s),\infty)$ and $q
\in(0,\infty]$, see \cite[Thm 1.2 and Remark 3.3]{KYZ}. For the definitions of ${\bf F}^s_{p,q}(\rn)$ and
${\bf B}^s_{p,q}(\rn)$, we refer to \cite{Tr}, \cite{Tr2}, \cite[Section 3]{KYZ}.
In the metric setting, $M^{s,p}(X)$ coincides with the Haj\l asz--Triebel--Lizorkin space $M_{p,\infty}^s(X)$, see \cite[Prop.\ 2.1]{KYZ} for a simple proof.

\subsection{Netrusov--Hausdorff content}
While studying the relation of Besov capacities and Hausdorff contents (sometimes called a classification problem), and Luzin type results for Besov functions, Netrusov used a modified version of the classical Hausdorff content in \cite{Ne92}, \cite{Ne96}. 
This content is used also for example in \cite{A2} and \cite{HN}. 
We use a slightly modified version where, instead of summing the powers of radii of the balls of the covering, we sum the measures of the balls divided by a power of the radii. This type of modification is natural in doubling metric spaces since the dimension of the space is not necessarily constant, not even locally. 

\begin{definition}\label{HN content}
Let $0\le d<\infty$, $0<\theta<\infty$  and $0<R<\infty$. 
The {\em Netrusov--Hausdorff content of codimension $d$} of a set $E\subset X$ is
\[
\cH^{d,\theta}_R
=\inf\Bigg[\sum_{i:2^{-i}\le R}\bigg(\sum_{j\in I_i} \frac{\mu(B(x_j,r_j))}{r_j^d} \bigg)^{\theta}\Bigg]^{1/\theta},
\]
where the infimum is taken over all coverings $\{B(x_j,r_j)\}$ of $E$ with $0<r_j\le R$ and $I_i=\{j: 2^{-i}\le r_j<2^{-i+1}\}$.

When $R=\infty$, the infimum is taken over all coverings of $E$ and the first sum is over $i\in \z$.
\end{definition}

Note that if measure $\mu$ is $Q$-regular, which means that the measure of each ball $B(x,r)$ is comparable with $r^Q$, then
$\cH^{d,\theta}_R$ is comparable with the $(Q-d)$-dimensional Netrusov--Hausdorff content defined using the powers of radii. 

A similar modification of the Hausdorff content, the {\em Hausdorff content of codimension $d$}, $0<d<\infty$, 
\[
\cH^{d}_R(E)=\inf\bigg\{\sum_{i=1}^{\infty} \frac{\mu(B(x_i,r_i))}{r_i^d}\bigg\},
\]
where $0<R<\infty$, and the infimum is taken over all coverings $\{B(x_i,r_i)\}$ of $E$ satisfying $r_i\le R$ for all $i$, has been used for example in the theory of BV-functions in metric spaces starting from \cite{Am}.  
When $R=\infty$, the infimum is taken over all coverings $\{B(x_i,r_i)\}$ of $E$, and the corresponding Hausdoff measure of codimension $d$ is
\[
\cH^d(E)=\lim_{R\to0}\cH^d_R(E).
\] 
Note that $\cH^{d,1}_R(E)=\cH^{d}_R(E)$ and by \eqref{a sum}, $\cH^{d,\theta}_R(E)\le \cH^{d}_R(E)$ if $\theta>1$ and 
$\cH^{d,\theta}_R(E)\ge \cH^{d}_R(E)$ if $\theta<1$.

\section{Lemmas}\label{lemmas}
This section contains lemmas needed in the proof of the main result
and new Poinc\'are type inequalities for $\gamma$-medians.

We start with an elementary inequality.
If $a_i\ge 0$ for all $i\in\z$ and $0<r\le 1$, then
\begin{equation}\label{a sum}
\Big(\sum _{i\in\z}a_i\Big)^r\le\sum_{i\in\z} a_i^r.
\end{equation}

The following two lemmas say that sets with small Netrusov--Hausdorff content have also small measure and 
that the content satisfies an Aoki--Rolewicz type estimate for unions even though it is not necessarily subadditive.

\begin{lemma}\label{measure and NH content}
Let $0<d,\theta,R<\infty$.
There is a constant $C>0$ such that 
 \[
 \mu(E)\le C\cH^{d,\theta}_R(E)
 \] 
for each measurable $E\subset X$. The claim holds also if $d=0$, $0<\theta\le1$ and $0<R\le\infty$.
\end{lemma}

\begin{proof}
We prove only the case $0<d,\theta,R<\infty$, the proof for the other case is similar. 
Let $\{B_j\}$ be a covering of $E$ by balls of radii $0<r_j\le R$. Then
\[
\mu(E)
\le \sum_{j=1}^\infty\mu(B_j)
\le \sum_{i:2^{-i}\le R}2^{(-i+1)d}\sum_{j\in I_i}\frac{\mu(B_j)}{r_j^d}.
\]
Hence
\[
\mu(E)\le CR^d\bigg(\sum_{i:2^{-i}\le R}\Big(\sum_{j\in I_i}\frac{\mu(B_j)}{r_j^d}\Big)^\theta\bigg)^{1/\theta}
\]
by \eqref{a sum} when $0<\theta\le 1$, and by the H\"older inequality for series when $\theta\ge 1$. 
The claim follows by taking infimum over coverings of $E$.
\end{proof}

\begin{lemma}\label{NH content subadd}
Let $0\le d<\infty$, $0<\theta<\infty$ and $0<R\le\infty$. 
Then, for all sets $E_k$, $k\in\n$ and for $r=\min\{1,\theta\}$,
\begin{equation}
\cH^{d,\theta}_R\Big(\bigcup_{k\in\n}E_k\Big)^r
\le \sum_{k\in\n}\cH^{d,\theta}_R(E_k)^{r}.
\end{equation}
\end{lemma}

\begin{proof}
Let $E=\cup_{k=1}^\infty E_k$. Let $\eps>0$.
For every $k$, let $B_{kj}$, $j\in\n$, be balls with radii $0<r_{kj}\le R$ such that $E_k\subset\cup_{j=1}^\infty B_{kj}$ and 
\[
\bigg(\sum_{i:2^{-i}\le R}\bigg(\sum_{j\in I^k_i} \frac{\mu(B_{kj})}{r_{kj}^d} \bigg)^{\theta}\bigg)^{r/\theta}
<\cH^{d,\theta}_R(E_k)^r+2^{-k}\eps,
\]
where $I^k_i=\{j: 2^{-i}\le r_{kj}<2^{-i+1}\}$. Then $\{B_{kj}: j,k\in\n\}$ is a covering of $E$
and so
\[
\begin{split}
\cH^{d,\theta}_R(E)^r
&\le \bigg(\sum_{i:2^{-i}\le R}\bigg(\sum_{k\in \n}\sum_{j\in I^k_i} \frac{\mu(B_{kj})}{r_{kj}^d} \bigg)^{\theta}\bigg)^{r/\theta}\\
&\le \sum_{k\in \n}\bigg(\sum_{i:2^{-i}\le R}\bigg(\sum_{j\in I^k_i} \frac{\mu(B_{kj})}{r_{kj}^d} \bigg)^{\theta}\bigg)^{r/\theta}\\
&\le \sum_{k\in\n}\cH^{d,\theta}_R(E_k)^r+\eps,
\end{split}
\]
where the second estimate comes from the fact that 
\[
\|\sum_{k\in\n}(a^k_i)_{i\in\z}\|_{l^\theta}^r
\le \sum_{k\in\n}\|(a^k_i)_{i\in\z}\|_{l^\theta}^r
\] 
for all $(a^k_i)_{i\in\z}\in l^\theta$, $k\in\n$. The claim follows by letting $\eps\to 0$.
\end{proof}

The following lemma is easy to prove using the H\"older inequality for series when $b\ge1$ and \eqref{a sum} when $0<b<1$. 
We need it while estimating the norms of fractional gradients. 

\begin{lemma}[\cite{HIT}, Lemma 3.1]\label{summing lemma}
Let $1<a<\infty$, $0<b<\infty$ and $c_k\ge 0$, $k\in\z$. There is a constant $C=C(a,b)$ such that
\[
\sum_{k\in\z}\Big(\sum_{j\in\z}a^{-|j-k|}c_j\Big)^b\le C\sum_{j\in\z}c_j^b.
\]
\end{lemma}

\begin{lemma}\label{B norm abs cont}
Let $0<p,q<\infty$ and $(g_k)\in l^q(L^p(X))$ or let $0<p<\infty$, $0<q\le \infty$ and $(h_k)\in L^p(l^q(X))$. 
Then $\|(g_k)\|_{l^q(L^p(\cdot))}$ and $\|(h_k)\|_{L^p(l^q(\cdot))}$ are absolutely continuous with respect to measure $\mu$. 
\end{lemma}

\begin{proof}
Let $\eps>0$.
Let $K\in\n$ be such that $\big(\sum_{|k|>K}\|g_k\|_{L^p(X)}^q\big)^{1/q}<\eps$.
By the absolute continuity of the $L^p$-norm,  there exists $\delta>0$ such that 
$\big(\sum_{|k|\le K}\|g_k\|_{L^p(A)}^q\big)^{1/q}<\eps$,
whenever $\mu(A)<\delta$. Hence, for such sets $A$ 
\[
\|(g_k)\|_{l^q(L^p(A))}=\Big(\sum_{k=-\infty}^\infty\|g_k\|_{L^p(A)}^q\Big)^{1/q}<C\eps, 
\]
from which the claim for $\|(g_k)\|_{l^q(L^p(\cdot))}$ follows.
For $\|(h_k)\|_{L^p(l^q(\cdot))}$, the claim follows by the absolute continuity of the $L^p$-norm.
\end{proof}

The next lemma contains basic properties of $\gamma$-medians. 
We leave the quite straightforward proof for the reader, who can also look at \cite{PT}.

\begin{lemma}\label{median lemma}
Let $A\subset X$ be a set with $\mu(A)<\infty$. 
Let $u,v\in L^0(A)$ and let $0<\gamma\le1/2$. The $\gamma$-median has the following properties:
\begin{enumerate}
\item\label{med subset}  If $A\subset B$ and there is $c>0$ such that $\mu(B)\le c\mu(A)$, 
then $m_{u}^{\gamma}(A)\le m_{u}^{\gamma/c}(B)$.
\item\label{med uv} If $u\le v$ almost everywhere in $A$, then $m_{u}^{\gamma}(A)\le m_{v}^{\gamma}(A)$.
\item\label{med gammas} If $0<\gamma_1\le\gamma_2\le1/2$, then $m_{u}^{\gamma_2}(A)\le m_{u}^{\gamma_1}(A)$.
\item\label{med constant} $m_u^\gamma(A)+c=m_{u+c}^\gamma(A)$ for each $c\in\re$.
\item\label{med abs}  $|m_{u}^\gamma(A)|\le m_{|u|}^{\gamma}(A)$. 
\item $m_{u+v}^\gamma(A)\le m_{u}^{\gamma/2}(A)+m_{v}^{\gamma/2}(A)$.
\item\label{med int} If $p>0$ and $u\in L^p(A)$, then
\[
m_{|u|}^\gamma(A)\le \Big(\gamma^{-1}\vint{A}|u|^p\,d\mu\Big)^{1/p}.
\]
\item\label{med leb point} $\lim_{r\to 0}m_{u}^\gamma(B(x,r))=u(x)$
for almost every $x\in X$.
\end{enumerate}
\end{lemma}

Property \eqref{med leb point} above says that medians of small balls behave like integral averages of locally integrable functions on Lebesgue points.
Recently, in \cite{HKT}, it was shown that for Haj\l asz--Besov and Haj\l asz--Triebel--Lizorkin functions, the
limit in \eqref{med leb point} exists outside a set of capacity zero.
Note also that if $u\in L^p(A)$, $p>0$, then by properties \eqref{med int} and \eqref{med leb point},
\begin{equation}\label{M iso}
u(x)\le M_\gamma u(x)\le \big(\gamma^{-1}\M u^p(x)\big)^{1/p} 
\end{equation}
for almost all $x\in X$. It follows from \eqref{M iso} and from the Hardy--Littlewood maximal theorem that, for every $p>0$, there exists a constant $C>0$ such that
\begin{equation}\label{Lp boundedness of median maximal op}
\|M_\gamma u\|_{L^p(X)}\le C\|u\|_{L^p(X)}
\end{equation}
for every $u\in L^p(X)$. 
More generally, \eqref{M iso} together with the Fefferman--Stein vector valued maximal theorem, proved in \cite{FS}, \cite{GLY}, \cite{Sa}, implies that, for every $0<p<\infty$ and $0<q\le\infty$, there exists a constant $C>0$ such that
\begin{equation}\label{Lplq boundedness of median maximal op}
\|(M_\gamma u_k)\|_{L^p(X;l^q)}\le C\|(u_k)\|_{L^p(X;l^q)}
\end{equation}
for every $(u_k)\in L^p(X;l^q)$.

\subsection{Poincar\'e type inequalities for medians}
The definition of the fractional $s$-gradient implies the validity of Poincar\'e type inequalities, which can be formulated using integral averages or in terms of medians. The versions for medians are extremely useful for functions that are not necessarily locally integrable. 
For integral versions, see \cite[Lemma 2.1]{KYZ} and \cite[Lemma 2.1]{GKZ}. 

\begin{theorem}\label{mm}
Let $0<\gamma\le 1/2$ and $0<s<\infty$. 
Let $u\in L^0(X)$ and $(g_k)\in \D^s(u)$. 
There exists a constant $C>0$, depending only on $s$ and $c_D$, such that inequality
\begin{equation}\label{median poincare}
\inf_{c\in\re}m_{|u-c|}^{\gamma}(B(x,2^{-i}))
\le C2^{-is}\sum_{k=i-3}^i m_{g_k}^{\gamma/C}(B(x,2^{-i+2}))
\end{equation}
holds for all $x\in X$, $i\in\z$. 
\end{theorem}

\begin{proof}
Let $x\in X$ and $i\in \z$.
Let $y\in B(x,2^{-i})$ and let $A=B(x,2^{-i+2})\setminus B(x,2^{-i+1})$.

Let $z\in A$. Then $2^{-i}\le d(z,y)<2^{-i+3}$ and hence
\[
|u(z)-u(y)|\le 2^{(-i+3)s}(g(z)+g(y)),
\] 
where
\[
g=\max_{i-3\le k\le i} g_k.
\]
By the nonempty spheres property, there is a point $a$ such that $d(a,x)=3\cdot 2^{-i}$. Then $B(a,2^{-i})\subset A$ and the doubling property implies that $\mu(A)\ge C\mu(B(x,2^{-i+2}))$.
Using Lemma \ref{median lemma}, we have, for $c=m_u^\gamma(A)$,
\[
\begin{split}
|u(y)-c|&\le m_{|u-u(y)|}^{\gamma}(A)\\
&\le C2^{-is}\big(m_{g}^{\gamma}(A)+g(y)\big)\\
&\le C2^{-is}\big(m_{g}^{\gamma/C}(B(x,2^{-i+2}))+g(y)\big),
\end{split}
\]
and hence
\[
\begin{split}
m_{|u-c|}^\gamma(B(x,2^{-i})) 
&\le C2^{-is}\big(m_{g}^{\gamma/C}(B(x,2^{-i+2}))+m_{g}^\gamma(B(x,2^{-i}))\big)\\
&\le C2^{-is}m_{g}^{\gamma/C}(B(x,2^{-i+2}))\\
&\le C2^{-is}\sum_{k=i-3}^i m_{g_k}^{\gamma/C}(B(x,2^{-i+2})),
\end{split}
\]
from which the claim follows.
\end{proof}

\begin{remark}
Inequalities \eqref{median poincare 1} and \eqref{pw med} mentioned in the introduction follow by similar, but even easier, arguments. 
\end{remark}

\begin{remark}
Recall that for a locally integrable function and a measurable set $A$ with $0<\mu(A)<\infty$, integral average $\vint{A}|u-u_A|\,d\mu$ is comparable with $\inf_{c\in\re}\vint{A}|u-c|\,d\mu$. Using Lemma \ref{median lemma}, it is easy to see that
\[
\inf_{c\in\re}\ m_{|u-c|}^\gamma(A)\le m_{|u-m_u^\gamma(A)|}^\gamma(A)\le 2\inf_{c\in\re}\ m_{|u-c|}^\gamma(A)
\]
for each measurable function $u$ and measurable set $A$ with finite measure.
\end{remark}

\begin{lemma}\label{pointwise lemma} 
Let $0<\gamma\le 1/2$ and $0<s<\infty$. 
Let $u\in L^0(X)$ and $(g_k)\in \D^s(u)$. Then

\begin{enumerate}
\item\label{mB1-mB2}
\[
|m_u^\gamma(B_1)-m_u^\gamma(B_2)|\le 2\inf_{c\in\re}m_{|u-c|}^{\gamma/{c_1}}(B_2)
\]
whenever $B_1$ and $B_2$ are balls such that $B_1\subset B_2$ and $\mu(B_2)\le c_1\mu(B_1)$,

\item\label{ux-mB}
\[
|u(x)-m_u^\gamma(B(y,2^{-i}))|\le C2^{-is}\sum_{k=i-4}^{i-1}M_{\gamma/C}\,{g_k}(x)
\]
for all $y\in X$, $i\in\z$ and for almost all $x\in B(y,2^{-i+1})$,
and 
\item \label{ux-uy}
\[
|u(x)-u(y)|
\le Cd(x,y)^s\sum_{k=i-4}^i\Big(M_{\gamma/C}\,{g_k}(x)+M_{\gamma/C}\,{g_k}(y)\Big)
\]
for almost all $x,y\in X$.
\end{enumerate}
\end{lemma}

\begin{proof}
\noindent \eqref{mB1-mB2}: Let $B_1$ and $B_2$ be as in the claim and let $c\in\re$. 
Using Lemma \ref{median lemma}, we have 
\[
|m_u^\gamma(B_1)-c|\le m_{|u-c|}^{\gamma}(B_1)\le m_{|u-c|}^{\gamma/C_1}(B_2)
\text{ and }|m_u^\gamma(B_2)-c|\le m_{|u-c|}^{\gamma}(B_2), 
\]
from which the claim follows using Lemma \ref{median lemma} and inequality
\[
|m_u^\gamma(B_1)-m_u^\gamma(B_2)|\le |m_u^\gamma(B_1)-c|+|m_u^\gamma(B_2)-c|.
\]

\noindent \eqref{ux-mB}: Let $x\in B(y,2^{-i+1})$ and let $A=B(y,2^{-i+3})\setminus B(y,2^{-i+2})$. 
Now
\[
|u(x)-m_u^\gamma(B(y,2^{-i}))|
\le |u(x)-m_u^\gamma(A)| + |m_u^\gamma(A)-m_u^\gamma(B(y,2^{-i}))|,
\]
and, by a similar argument as in the proof of Theorem \ref{mm},
\[
|u(x)-m_u^\gamma(A)|
\le C2^{-is}\big(m_{g}^{\gamma/C}(B(y,2^{-i+3}))+g(x)\big),
\]
where $g=\max_{i-4\le k\le i-2} g_k$. 
Hence, by the fact that $B(y,2^{-i+3})\subset B(x,2^{-i+4})$, Lemma \ref{median lemma} and \eqref{M iso}, we have
\[
|u(x)-m_u^\gamma(A)|
\le C2^{-is}\sum_{k=i-4}^{i-2}M_{\gamma/C}\,{g_k}(x).
\]
The claim follows since, by Lemma \ref{median lemma}, a similar argument as in the proof of Theorem \ref{mm}, and \eqref{M iso},
\[
\begin{split}
|m_u^\gamma(A)-m_u^\gamma(B(y,2^{-i}))|
&\le m_{|u-m^\gamma_u(A)|}(B(y,2^{-i}))\\
&\le C 2^{-is} \sum_{k=i-4}^{i-2}m_{g_k}^{\gamma/C}(B(y,2^{-i+3}))\\
&\le C 2^{-is} \sum_{k=i-4}^{i-2}M_{\gamma/C}g_k(x).
\end{split}
\]

\noindent \eqref{ux-uy}: Let $x,y\in X$ and let $i\in\z$ be such that $2^{-i-1}<d(x,y)\le 2^{-i}$. 
Then 
\begin{align*}
|u(x)-u(y)|
&\le |u(x)-m_u^\gamma(B(x,2^{-i}))|+|m_u^\gamma(B(x,2^{-i}))-m_u^\gamma(B(y,2^{-i}))|\\
&\quad +|u(y)-m_u^\gamma(B(y,2^{-i}))|,
\end{align*}  
and the claim follows using \eqref{mB1-mB2} and \eqref{ux-mB}, Theorem \ref{mm} and \eqref{M iso}.
\end{proof}

\begin{remark}\label{medians are better}
Lemma \ref{pointwise lemma} \eqref{ux-mB} and Lemma \ref{median lemma} imply that, for every $t>0$, there exists a constant $C>0$ such that
\begin{equation}\label{med ie}
|u(x)-m_u^\gamma(B(y,2^{-i}))|\le C2^{-is}\sum_{k=i-4}^{i-1}\big(\M g_k^t(x)\big)^{1/t}
\end{equation}
for all $y\in X$, $i\in\z$ and for almost all $x\in B(y,2^{-i+1})$. 
For integral averages, only a weaker estimate 
\begin{equation}\label{av ie}
|u(x)-u_{B(y,2^{-i})}|\le C2^{-is}\sum_{k=i-4}^\infty 2^{(i-k)s'}\big(\M g_k^t(x)\big)^{1/t},
\end{equation}
where $t>Q/(Q+s)$ and $0<s'<s$, is known to hold.
The estimate \eqref{av ie} can be proven using a chaining argument and a Sobolev--Poincar\'e type inequality from \cite{GKZ}, see \cite[Lemma 4.2]{HKT}. Somewhat surprisingly, with medians, a better estimate \eqref{med ie} follows by a completely elementary argument.
\end{remark}

The pointwise estimates of Lemma \ref{pointwise lemma2} in terms of the fractional sharp median maximal function are needed in the proof of the H\"older continuity of the approximating function in Theorem \ref{holder representative B}.

\begin{definition}
Let $0<\gamma\le1/2$, $\beta>0$, $R>0$. Let $u\in L^0(X)$.
The {\em (restricted, uncentered) fractional sharp $\gamma$-median maximal function} of $u$ is 
$\tilde u^{\gamma,\#}_{\beta, R}\colon X\to[0,\infty]$,
\[
\tilde u^{\gamma,\#}_{\beta, R}(x)
=\sup_{0<r\le R,\,x\in B(y,r)}\,r^{-\beta} \inf_{c\in\re}m_{|u-c|}^\gamma(B(y,r)),
\]
and the {\em (restricted) fractional sharp $\gamma$-median maximal function} is
$u^{\gamma,\#}_{\beta, R}\colon X\to[0,\infty]$,
\[
u^{\gamma,\#}_{\beta, R}(x)
=\sup_{0<r\le R}\,r^{-\beta} \inf_{c\in\re}m_{|u-c|}^\gamma(B(x,r)).
\] 
The unrestricted versions $u^{\gamma,\#}_{\beta,\infty}$, $\tilde u^{\gamma,\#}_{\beta,\infty}$ are denoted shortly by 
$u^{\gamma,\#}_{\beta}$ and $\tilde u^{\gamma,\#}_{\beta}$.
\end{definition}

It follows easily from the definitions and Lemma \ref{median lemma} that 
\begin{equation}\label{maximal functions comparable}
u^{\gamma,\#}_{\beta, R}(x)\le \tilde u^{\gamma,\#}_{\beta, R}(x)
\le 2^\beta u^{\gamma/c_D,\#}_{\beta, 2R}(x).
\end{equation}

\begin{lemma}\label{pointwise lemma2} 
Let $0<\gamma\le 1/2$ and $\beta>0$. 
Let $u\in L^0(X)$.

\begin{enumerate}
\item\label{mB1-mB2 sharp}
If $B_1$ and $B_2=B(x,r)$ are balls such that $B_1\subset B_2$ and $\mu(B_2)\le c_1\mu(B_1)$, then
\[
|m_u^\gamma(B_1)-m_u^\gamma(B_2)|
\le Cr^\beta u^{\gamma/{c_1},\#}_{\beta,r}(x).
\]

\item\label{ux-mB sharp}
If $y\in X$ and $r>0$, then
\[
|u(x)-m_u^\gamma(B(y,r))|\le Cr^\beta u^{\gamma/{C},\#}_{\beta,Cr}(x)
\]
for almost all $x\in B(y,2r)$.

\item \label{ux-uy sharp}
Inequality
\[
|u(x)-u(y)|\le Cd(x,y)^\beta\Big(u^{\gamma/{C},\#}_{\beta,3d(x,y)}(x)+u^{\gamma/{C},\#}_{\beta,3d(x,y)}(y)\Big)
\]
holds for almost all $x,y\in X$.
\end{enumerate} 
\end{lemma}
 
\begin{proof}
Inequality \eqref{mB1-mB2 sharp} follows from Lemma \ref{pointwise lemma} \eqref{mB1-mB2}.

To prove \eqref{ux-mB sharp}, let $x$ be such that Lemma \ref{median lemma} \eqref{med leb point} holds
and let $c\in\re$. Then
\[
|u(x)-m_u^\gamma(B(y,r))|
\le |u(x)-m_u^\gamma(B(x,r))| + |m_u^\gamma(B(x,r))-m_u^\gamma(B(y,r))|,
\]
where, by a telescoping argument and Lemma \ref{pointwise lemma} \eqref{mB1-mB2},
\begin{equation}\label{ux-mBx sharp}
\begin{aligned}
|u(x)-m_u^\gamma(B(x,r))|
&\le \sum_{j=0}^\infty  |m_u^\gamma(B(x,2^{-j-1}r))-m_u^\gamma(B(x,2^{-j}r))|\\
& \le 2\sum_{j=0}^\infty m_{|u-c|}^{\gamma/c_D}(B(x,2^{-j}r)\\
&\le Cr^\beta u^{\gamma/{C},\#}_{\beta,r}(x).
\end{aligned}
\end{equation}
Since $B(y,r)\subset B(x,3r)$ with comparable measures, Lemma \ref{pointwise lemma} \eqref{mB1-mB2} shows that 
\[
|m_u^\gamma(B(x,r))-m_u^\gamma(B(y,r))|
\le 4 m_{|u-c|}^{\gamma/c_D^2}(B(x,3r))
\le Cr^\beta u^{\gamma/{C},\#}_{\beta,3r}(x),
\]
and the claim follows. 

For \eqref{ux-uy sharp}, let $x,y\in X$ be such that Lemma \ref{median lemma} \eqref{med leb point} holds.
Then 
\[
|u(x)-u(y)|
\le |u(x)-m_u^\gamma(B(x,d(x,y)))|+|m_u^\gamma(B(x,d(x,y)))-u(y)|,
\]
and the claim follows using \eqref{ux-mB sharp}.
\end{proof}

The following Leibniz rule for a function having a fractional $s$-gradient and a bounded, compactly supported Lipschitz function has been proved in \cite[Lemma 3.10]{HIT}. 
To prove the norm estimates of lemma below, $s$-gradient $(g'_k)_{k\in\z}$, 
\[
g'_k=
\begin{cases}
h_k,\quad &\text{if }k<k_L,\\
\rho_k,\quad &\text{if }k\ge k_L,
\end{cases}
\]
where $k_L$ is an integer such that $2^{k_L-1}<L\le 2^{k_L}$, is used for $u\ph$.

\begin{lemma}\label{leibniz}
Let $0<s<1$, $0<p<\infty$, $0<q\le\infty$, $L>0$, and let $S\subset X$ be a measurable set.
Let $u\colon X\to\re$ be a measurable function with $(g_k)\in \D^s(u)$ and
let $\ph$ be a bounded $L$-Lipschitz function supported in $S$.
Then sequences $(h_k)_{k\in\z}$ and $(\rho_k)_{k\in\z}$, where
\[
\rho_k=\big(g_k\Vert \ph\Vert_{\infty}+2^{k(s-1)}L|u|\big)\ch{\operatorname{supp}\ph}
 \quad\text{and}\quad
h_k=\big(g_k+2^{s k+2}|u|\big)\Vert \ph\Vert_{\infty}\ch{\operatorname{supp}\ph}
\]
are fractional $s$-gradients of $u\ph$.
Moreover, if $u\in N^s_{p,q}(S)$, then $u\ph\in N^s_{p,q}(X)$ and $\|u\ph\|_{N^s_{p,q}(X)}\le C\|u\|_{N^s_{p,q}(S)}$. 
A similar result holds for functions in $M^s_{p,q}(S)$.
\end{lemma}

\section{Proof of Theorems \ref{holder representative B} and  \ref{holder representative TL} - H\"older approximation}\label{proofs}

In the proof, we use the representative $\tilde u$,
\begin{equation}\label{u mato}
\tilde u(x)=\limsup_{r\to 0}m_u(B(x,r))
\end{equation}
for $u$ and denote it by $u$. 
By Lemma \ref{median lemma}, the limit of \eqref{u mato} exists and equals $u(x)$, except on a set of zero measure.
Since, by Lemma \ref{pointwise lemma2}, inequality 
\begin{equation}\label{u pointwise}
|u(x)-u(y)|
\le C d(x,y)^\beta\Big(u_{\beta,3 d(x,y)}^{\gamma/C,\#}(x) + u_{\beta,3 d(x,y)}^{\gamma/C,\#}(y)\Big)      
\end{equation}
holds for every $x,y\in X$ and for all $0<\beta\le 1$, $u$ is $\beta$-H\"older continuous 
if $\|u_{\beta}^{\gamma/C,\#}\|_\infty<\infty$. 
\smallskip

We will first assume that $u$ vanishes outside a ball. 
The general case follows using a localisation argument.
We will correct the function in ``the bad set'', where the fractional sharp median maximal function is large, using a discrete convolution. 
This kind of smoothing technique is used to prove corresponding approximation results for Sobolev functions on metric measure spaces in \cite[Theorem 5.3]{HjKi} and \cite[Theorem 5]{KiTu}.  Since we use medians instead of integral averages in the discrete convolution, the proof works also for $0<p,q\le 1$.

For the bad set, we will use a Whitney type covering from \cite[Theorem III.1.3]{CW}, 
\cite[Lemma 2.9]{MS}. For each open set $U\subset X$, there are balls
$B_i=B(x_i,r_i)$, $i\in\n$, where $r_i=\dist(x_i,X\setminus U)/10$,
such that
\begin{enumerate}
\item\label{whitney disjoint}
the balls $1/5B_i$ are disjoint,  
\item\label{whitney cover}
$U=\cup_{i\in\n}B_i$,  
\item\label{whitney inside U}
$5B_i\subset U$ for each $i$,
\item\label{whitney distance to U}
if $x\in 5B_i$, then $5r_i\le\dist(x,X\setminus U)\le15r_i$,  
\item\label{whitney x*i}
for each $i$, there is $x^*_i\in X\setminus U$ such that $d(x_i,x^*_i)<15r_i$, and 
\item\label{whitney overlap}
$\sum_{i=1}^{\infty}\ch{5B_i}\le C\ch{U}$.
\end{enumerate}
Corresponding to a Whitney covering, there is a sequence $(\ph_i)_{i\in\n}$ of $K/r_i$-Lipschitz functions, called a partition of unity, 
such that $\operatorname{supp}\ph_i\subset 2B_i$, $0\le\ph_i\le1$, and $\sum_{i=1}^{\infty}\ph_i=\ch{U}$, 
see for example \cite[Lemma 2.16]{MS}.

\begin{proof}[Proof of Theorem \ref{holder representative B}]
 Let $u\in N^s_{p,q}(X)$ and let $(g_k)_{k\in\z}$ be a fractional $s$-gradient of $u$ such that $\|(g_k)\|_{l^q(L^p(X))}\le 2\|u\|_{N^s_{p,q}(X)}$. Let $0<\beta\le s$.
\smallskip

\noindent{\sc Step 1:} Assume that $u$ is supported in $B(x_0,1)$ for some $x_0\in X$. 

Let $0<\gamma\le1/2$ and $\lambda>0$. 
We will modify $u$ in set 
\[
E_{\lambda}=\big\{x\in X:\tilde u^{\gamma/c_E,\#}_\beta(x)>\lambda\big\},
\]
where $c_E$ will be the largest constant in the fractional sharp median maximal functions of $u$ below in the proof of the H\"older continuity of $v$. 
We need a Whitney covering $\{B_i\}_i$ of $E_{\lambda}$
and a corresponding partition of unity $(\ph_i)_i$.
For each $x_i$, let $x^*_i$ be the ``closest'' point in $X\setminus E_{\lambda}$.

We begin with the properties of the set $E_\lambda$. 
It follows directly from the definition that $E_{\lambda}$ is open. 
By \eqref{maximal functions comparable}, $E_\lambda\subset\big\{x\in X:u^{\gamma/(c_Dc_E),\#}_\beta(x)>2^{-\beta}\lambda\big\}$,
and \eqref{u pointwise} shows that $u$ is $\beta$-H\"older continuous in $X\setminus E_{\lambda}$ . 
\smallskip

\noindent {\bf Claim 1:} There is $\lambda_0>0$ such that 
$E_{\lambda}\subset B(x_0,2)$ for each $\lambda>\lambda_0$. 

\begin{proof}
Since $\operatorname{supp} u\subset B(x_0,1)$, by \eqref{maximal functions comparable}, it suffices to show that there is 
$\lambda_0>0$ such that 
\begin{equation}\label{lambda0}
r^{-\beta}m^{\gamma/(c_Ec_D)}_u(B(x,r))<\lambda_0
\end{equation}
for all $x\in X$ and $r>1$. 
If $B=B(x,r)$, $r>1$ and $r^{-\beta}m^{\gamma/(c_Ec_D)}_u(B(x,r))=a>0$,
then $B\cap B(x_0,1)\ne\emptyset$ because $\operatorname{supp} u\subset B(x_0,1)$. 
Using Lemma \ref{median lemma} and the doubling property of $\mu$, we obtain
\begin{align*}
r^{-\beta}m^{\gamma/(c_Ec_D)}_u(B(x,r))
&\le\bigg(\frac{c_Ec_D}\gamma\vint{B}|u|^p\,d\mu\bigg)^{1/p}\\
&\le\bigg(\frac{c_Ec_D}\gamma\mu(B(x_0,1))^{-1}\int_{B(x_0,1)}|u|^p\,d\mu\bigg)^{1/p}\\
&\le\bigg(\frac{c_Ec_D}\gamma\bigg)^{1/p}\mu(B(x_0,1))^{-1/p}\|u\|_{L^p(X)},
\end{align*}
from which Claim 1 follows. 
\end{proof}

\noindent{\bf Claim 2:} There is a constant $R>0$, independent of $u$ and the parameters of the theorem, such that 
$\cH^{(s-\beta)p,q/p}_R(E_\lambda)\to0$ as $\lambda\to\infty$.

\begin{proof}
We will show that 
\begin{equation}\label{H Elambda}
\cH^{(s-\beta)p,q/p}_R(E_\lambda)\le C\lambda^{-p}\|u\|_{\dot N^s_{p,q}(X)}^p,
\end{equation}
where the constant $C>0$ is independent of $u$ and $\lambda$.
\medskip

Let $x\in E_\lambda$ and $\lambda>\lambda_0$. 
Let $r>0$ and let $l\in\z$ be such that $2^{-l-1}\le r<2^{-l}$.
Using the doubling condition, Theorem \ref{mm}, and Lemma \ref{median lemma}, we obtain
\begin{align*}
\inf_{c\in\re}m_{|u-c|}^{\gamma/C}(B(x,r))
&\le C2^{-ls}\sum_{k= l-3}^l m_{g_k}^{\gamma/C}(B(x,2^{-l+2}))\\
&\le C2^{-l\beta}\sum_{k= l-3}^l\;\Big(\,\frac{C}{\gamma}2^{-l(s-\beta)}\vint{B(x,2^{-l+2})}g_k^p\,d\mu\Big)^{1/p},
\end{align*}
which implies, by \eqref{lambda0}, that
\[
u^{\gamma/C,\#}_{\beta}(x)=u^{\gamma/C,\#}_{\beta,1}(x)
\le C\sup_{i\ge-1}2^{-i(s-\beta)}\Big(\,\vint{B(x,2^{-i+5})}g_i^p\,d\mu\Big)^{1/p}. 
\]
Hence
\[
E_\lambda\subset
\bigg\{x\in X: C\sup_{i\ge-1}2^{-i(s-\beta)}\Big(\,\vint{B(x,2^{-i+5})}g_i^p\,d\mu\Big)^{1/p} >\lambda\bigg\}=:F_\lambda.
\]
By the standard $5r$-covering lemma, there are disjoint balls $B_j$, $j\in\n$, of radii $r_j\le R$ with $R=2^6$ such that the balls 
$5B_j$ cover $F_\lambda$ and
\[
\mu(B_j){r_j}^{-(s-\beta)p}<C\lambda^{-p}\int_{B_j}g_{i+5}^p\,d\mu
\]
for $j\in I_i$.
Using the disjointness of the balls $B_j$, we have
\[
\sum_{j\in I_i} \frac{\mu(B_j)}{r_j^{(s-\beta)p}}\le C\lambda^{-p} \|g_{i+5}\|_{L^p(X)}^p
\]
for every $i\in\z$, which implies that
\[
\cH^{(s-\beta)p,q/p}_R(F_\lambda)\le C\lambda^{-p}\bigg(\sum_{i\in\z}\|g_i\|_{L^p(X)}^q\bigg)^{p/q}.
\]
Hence the claim follows. 
\end{proof}

Note that, by Claim 2 and Lemma \ref{measure and NH content}, $\mu(E_\lambda)\to0$ as $\lambda\to0$.
\smallskip

\noindent{\bf Extension to $E_{\lambda}$: }
We define $v$, a candidate for the approximating function, as a Whitney type extension of $u$ to $E_{\lambda}$,
\[
v(x)=
\begin{cases}
u(x),&\text{ if }x\in X\setminus E_{\lambda},\\
\sum_{i=1}^{\infty}\ph_i(x)m^\gamma_u(2B_i),&\text{ if }x\in E_{\lambda},
\end{cases}  
\]
and select the open set $\Omega$ to be $E_\lambda$ for sufficiently large $\lambda>\lambda_0$. 
Hence property \eqref{holder 1} of Theorem 
\ref{holder representative B} follows from the definition of $v$ and property \eqref{holder 4} from Claim 3.
Since $\operatorname{supp} u\subset B(x_0,1)$ and $E_{\lambda}\subset
B(x_0,2)$ for $\lambda>\lambda_0$, the support of $v$ is in $B(x_0,2)$.
By the bounded overlap of the balls $2B_i$, there is a bounded number of indices in
\[
I_x=\{i:x\in2B_i\}
\]  
for each $x\in E_{\lambda}$, and the bound is independent of $x$.
\smallskip

Next we prove an estimate for $|v(x)-v(\bar x)|$, where
$x\in E_{\lambda}$ and $\bar x\in X\setminus E_{\lambda}$ is such
that $d(x,\bar x)\le 2\dist(x,X\setminus E_{\lambda})$. 
Using the properties of the functions $\ph_i$, we have that
\begin{equation}\label{vx-vx}
|v(x)-v(\bar x)|
=\Big|\sum_{i=1}^{\infty}\ph_i(x)(u(\bar x)-m^\gamma_u(2B_i))\Big|
\le \sum_{i\in I_x}|u(\bar x)-m^\gamma_u(2B_i)|,
\end{equation}
where, by the fact that $2B_i\subset B(\bar x,Cr_i)$ and $B(\bar x,Cr_i)\subset CB_i$ for all $i\in I_x$, and 
by Lemma \ref{pointwise lemma2}, 
\begin{equation}\label{u-u2B}
\begin{aligned}
|u(\bar x)-m^\gamma_u(2B_i)|
&\le |u(\bar x)-m^\gamma_u(B(\bar x,Cr_i))|+|m^\gamma_u(B(\bar x,Cr_i))-m^\gamma_u(2B_i)|\\
&\le cr_i^{\beta}u_{\beta,Cr_i}^{\gamma/C,\#}(\bar x). 
\end{aligned}
\end{equation}
Since $r_i\approx\dist(x,X\setminus E_{\lambda})$, 
estimates \eqref{vx-vx}-\eqref{u-u2B} show that
\begin{equation}\label{vx-vx2}
|v(x)-v(\bar x)|
\le c\dist(x,X\setminus E_{\lambda})^{\beta}u^{\gamma/C,\#}_{\beta}(\bar x)  
\le c\lambda d(x,\bar x)^{\beta}.
\end{equation}

\noindent{\bf Proof of \eqref{holder 2} - the H\"older continuity of $v$:}
We will show that
\begin{equation}\label{s holder}
|v(x)-v(y)|\le c\lambda d(x,y)^{\beta}\text{ for all }x,y\in X.
\end{equation}

\noindent (i) If $x,y\in X\setminus E_{\lambda}$, then \eqref{s holder} 
follows from \eqref{u pointwise} and the definition of $E_{\lambda}$. 

\noindent (ii) Let $x,y\in E_{\lambda}$ and $d(x,y)\le M$, where
\[
M=\min\bigl\{\dist(x,X\setminus E_{\lambda}),
               \dist(y,X\setminus E_{\lambda})\bigr\}.
\]
Let $\bar x\in X\setminus E_{\lambda}$ and sets $I_x$ and $I_y$ be as above. 
We may assume that $\dist(x,X\setminus E_{\lambda})\le \dist(y,X\setminus E_{\lambda})$.
Then
\[
\dist(y,X\setminus E_{\lambda})\le d(x,y)+\dist(x,X\setminus E_{\lambda})
\le 2\dist(x,X\setminus E_{\lambda}),
\]
and hence $r_i$ is comparable to $M$ for all $i\in  I_x\cup I_y$.

By the properties of the functions $\ph_i$, we have
\begin{align*}
|v(x)-v(y)|
&=\Big| \sum_{i=1}^\infty \big(\ph_i(x)-\ph_i(y)\big) \big(u(\bar x)-m^\gamma_u(2B_i)\big)\Big|\\
&\le c d(x,y)\sum_{i\in I_x\cup I_y}r_i^{-1}|u(\bar x)-m^\gamma_u(2B_i)|.
\end{align*}
Hence, using a similar argument as for \eqref{u-u2B}, the fact that there are a bounded number of indices in $I_x\cup I_y$, and the assumption $M\ge d(x,y)$, we obtain
\[
|v(x)-v(y)|
\le c d(x,y)\sum_{i\in I_x\cup I_y} r_i^{\beta-1}u^{\gamma/C,\#}_\beta(\bar x)
\le c d(x,y)^\beta\lambda.
\]

\noindent (iii) Let $x,y\in E_{\lambda}$ and $d(x,y)>M$. 
Let $\bar x, \bar y\in X\setminus E_{\lambda}$ be as above. 
Using inequalities \eqref{vx-vx2} and \eqref{u pointwise}, we have
\begin{align*}
|v(x)-v(y)|
&=|v(x)-v(\bar x)|+|v(\bar x)-v(\bar y)|+|v(\bar y)-v(y)|\\
&\le c\lambda \big(d(x,\bar x)^\beta+d(\bar x,\bar y)^\beta+d(y,\bar y)^\beta\big)
\le c\lambda d(x,y)^\beta.
\end{align*}

\noindent (iv) Let $x\in E_\lambda$ and $y\in X\setminus E_\lambda$. Then, by \eqref{vx-vx2} and \eqref{u pointwise},
\begin{align*}
|v(x)-v(y)|
&=|v(x)-v(\bar x)|+|u(\bar x)-u(y)|
\le c\lambda d(x,y)^\beta.
\end{align*}
The H\"older continuity of $v$ with estimate \eqref{s holder} follows from the four cases above. 
Now we select $c_E$ in the definition of $E_\lambda$ to be the maximum of the constants $C$ in $\lambda/C$'s in the fractional median maximal functions in the proof above and \eqref{u pointwise}.
\medskip

\noindent{\bf Proof of \eqref{holder 3} - a fractional $s$-gradient for $v$:}
\begin{lemma}\label{sgrad for v}
There is a constant $C>0$ such that $(C\tilde g_k)_{k\in\z}$, where 
\begin{equation}\label{eq:gradient}
\tilde g_k=\sup_{j\in\z}2^{-|j-k|\delta}M_{\gamma/C}\,g_j
\end{equation}
and $\delta=\min\{s,1-s\}$, is a fractional $s$-gradient of $v$.
\end{lemma}

\begin{proof}
Since every fractional $s$-gradient of $u$ is a fractional $s$-gradient of $|u|$, we may assume that $u\ge 0$. 

Let $k\in\z$ and let $x,y\in X$ such that $2^{-k-1}\le d(x,y)<2^{-k}$.
We will show that
\begin{equation}\label{vgk}
|v(x)-v(y)|\le Cd(x,y)^{s}(\tilde g_k(x)+\tilde g_k(y)),
\end{equation}
where constant $C>0$ is independent of $k$. 
For each $x\in E_\lambda$, let set $I_{x}$ be as earlier. 
We consider the following four cases:

\noindent
{\sc Case 1:} 
Since $g_{k}\le\tilde{g_{k}}$ almost everywhere on $X\setminus E_\lambda$, we have, for almost every 
$x,y\in X\setminus E_\lambda$,
\[
|v(x) - v(y)|
=| u(x)-u(y)|
\le d(x,y)^{s}(\tilde g_{k}(x)+\tilde g_{k}(y)).
\]
\noindent
{\sc Case 2:} $x\in X\setminus E_\lambda$ and $y\in E_\lambda$. 
Let $R=d(y,X\setminus E_\lambda)$ and let $l$ be such that $2^{-l-1}<R\le 2^{-l}$. 
Then we have
\begin{align*}
|v(x)-v(y)|&= |u(x)-v(y)|
=\Big|\sum_{i=1}^{\infty}\ph_i(y)(u(x)-m^\gamma_u(2B_i))\Big|\\
&\le C|u(x)-u(y)|+\sum_{i\in I_y}|u(y)-m^\gamma_u(2B_i)|,
\end{align*}
where the desired estimate for the first term follows as in Case 1.
For the second term, using the properties of the functions $\ph_{i}$, the fact that $2B_i\subset B(y,R)$ with comparable radius for all 
$i\in I_y$, the doubling property, Lemma \ref{pointwise lemma}, Theorem \ref{median poincare}, and the facts that there are bounded number of indices in $I_{y}$ and $k\le l$, we obtain
\begin{equation}\label{uy-u2Bi}
\begin{aligned}
\sum_{i\in I_y}|u(y)-m^\gamma_u(2B_i)|
&\le\sum_{i\in I_y}\Big(|u(y)-m^\gamma_u(B(y,R))|+|m^\gamma_u(B(y,R))-m^\gamma_u(2B_i)|\Big)\\
&\le C2^{-ls}\sum_{j=l-4}^l M_{\gamma/C}\,g_j(y)\\
&\le C2^{-ks}\sup_{j\ge k} 2^{(k-j)s}M_{\gamma/C}\,g_j(y)\\
&\le Cd(x,y)^s\tilde g_k(y).
\end{aligned}
\end{equation}
Hence inequality \eqref{vgk} follows in this case.
\smallskip

\noindent
{\sc Case 3:} $x,y\in E_\lambda$, $d(x,y)\le M$, where
\[
M=\min\bigl\{\dist(x,X\setminus E_{\lambda}),
               \dist(y,X\setminus E_{\lambda})\bigr\}.
\]
We may assume that $\dist(x,X\setminus E_{\lambda})\le \dist(y,X\setminus E_{\lambda})$.
Then $\dist(y,X\setminus E_{\lambda})\le 2\dist(x,X\setminus E_{\lambda})$, 
$r_i$ is comparable to $M$ 
and $2B_i\subset B(x,4M)$ for all $i\in  I_x\cup I_y$.

Let $l$ be such that $2^{-l-1}<4M\le 2^{-l}$. 
Using the doubling condition, the properties of the functions $\ph_{i}$,
the fact that there are bounded number of indices in $I_{x}\cup I_{y}$ and 
Lemma \ref{pointwise lemma} and Theorem \ref{median poincare}, we have 
\begin{align*}
 |v(x)-v(y)|
&\le \sum_{i=1}^\infty|\ph_i(x)-\ph_i(y))||m^\gamma_u(2B_i)-m^\gamma_u(B(x,2^{-l}))|\\
&\le Cd(x,y)\sum_{i\in I_x\cup I_y} r_i^{-1}|m^\gamma_u(2B_i)-m^\gamma_u(B(x,2^{-l}))|\\
&\le Cd(x,y)\sum_{i\in I_x\cup I_y} r_i^{-1}2^{-ls}\sum_{j=l-3}^l m_{g_j}^{\gamma/C}(B(x,2^{-l+2}))\\
&\le Cd(x,y)M^{-1}2^{-ls}\sum_{j=l-3}^l M_{\gamma/C}\,g_j(x),
\end{align*}
where, since $M\approx 2^{-l}$ and $d(x,y)<2^{-k}$, 
\[
d(x,y)M^{-1}2^{-ls}
\le Cd(x,y)^sd(x,y)^{1-s}2^{l(1-s)}
\le Cd(x,y)^s\ 2^{(l-k)(1-s)}.
\]
Hence
\[
\begin{split}
|v(x)-v(y)|
\le \ &Cd(x,y)^{s}
\sup_{j\le k} 2^{(j-k)(1-s)}M_{\gamma/C}\,g_j(x)\\
\le \ &Cd(x,y)^{s}\tilde g_k(x),
\end{split}
\]
which implies the claim in this case. 
\medskip

\noindent
{\sc Case 4:} $x,y\in E_\lambda$, $d(x,y)> M$. Now
\begin{align*}
|v(x)-v(y)|
&\le\sum_{i\in I_x}|u(x)-m^\gamma_u(2B_i)|+\sum_{i\in I_y}|u(y)-m^\gamma_u(2B_i)|+|u(x)-u(y)|,
\end{align*}
and the claim follows using the properties of the functions $\ph_i$, similar estimates for $|u(x)-m^\gamma_u(2B_i)|$ and
$|u(y)-m^\gamma_u(2B_i)|$ as in \eqref{uy-u2Bi}, and the fact that $r_i$ is comparable to $\dist(x,X\setminus E_{\lambda})$ for all 
$i\in I_x$ (and similarly for $I_y$).
\end{proof}

\noindent{\bf Proof of \eqref{holder 3} - $v\in N^{s}_{p,q}$ and approximation in norm:} 
\begin{lemma}\label{IntForGradOfExtensionBesov}
$\|(\tilde g_k)\|_{l^q(L^p(X))}\le C\|(g_k)\|_{l^q(L^p(X))}$.
\end{lemma}
\begin{proof}
By \eqref{Lp boundedness of median maximal op},
\[
\|\tilde g_k\|_{L^p(X)}^p
\le \sum_{j\in\z}2^{-|j-k|\delta p}\|M_{\gamma/C}\,g_j\|_{L^p(X)}^p\le C\sum_{j\in\z}2^{-|j-k|\delta p}\|g_j\|_{L^p(X)}^p.
\]
Hence, by Lemma \ref{summing lemma}, we obtain
\[
\sum_{k\in\z}\|\tilde g_k\|_{L^p(X)}^q
\le C\sum_{k\in\z}\Big(\sum_{j\in\z}2^{-|j-k|\delta p}\|g_j\|_{L^p(X)}^p\Big)^{q/p}
\le C\sum_{j\in\z}\|g_j\|_{L^p(X)}^q,
\]
which gives the claim.
\end{proof}

By the properties of the Whitney covering and Lemma \ref{median lemma}, we have
\begin{align*}
|v(x)|
&\le \sum_{i=1}^\infty \ph_i(x)m_{u}^\gamma(2B_i)
\le C\sum_{i\in I_x} \ph_i(x)M_{\gamma/C}(u\ch{E_\lambda})(x)\\
&\le CM_{\gamma/C}(u\ch{E_\lambda})(x)
\end{align*}
for each $x\in E_{\lambda}$, and, by the boundedness of $M_{\gamma/C}$ on $L^{p}$, that  
\begin{equation}\label{norm of v}
\begin{aligned}
\|v\|_{L^p(X)}^p
&\le \|u\|_{L^p(X\setminus E_{\lambda})}^p
+ C\| M_{\gamma/C}(u\ch{E_\lambda})\|_{L^p(X)}^p
\le C\|u\|_{L^p(X)}^p.
\end{aligned}
\end{equation}
Since $v=u$ in $X\setminus E_{\lambda}$, we have that
\[
\|u-v\|_{L^p(X)}=\|u-v\|_{L^p(E_\lambda)},
\]
which tends to $0$ as $\lambda\to\infty$ because $\mu(E_{\lambda})\to0$ as $\lambda\to\infty$. Hence $v\to u$ in $L^p(X)$.
\smallskip

\noindent {\bf Claim:} Sequence $(h_k^\lambda)_{k\in\z}$, where
\[
h_k^\lambda=\tilde g_k\ch{E_\lambda},
\]
is a fractional $s$-gradient of $u-v$ and $\|(h_k^\lambda)\|_{l^q(L^p(X))}\to 0$ as $\lambda\to\infty$. 

\begin{proof}
We have to show that inequality
\begin{equation}\label{h sgrad}
 |(u-v)(x)-(u-v)(y)|
 \le Cd(x,y)^{s}(h_k^\lambda(x)+h_k^\lambda(y))
\end{equation}
holds outside a set of measure zero whenever $2^{-k-1}\le d(x,y)<2^{-k}$.

If $x,y\in X\setminus E_\lambda$, then $u-v=0$ and \eqref{h sgrad} holds.
If $x,y\in E_\lambda$, then inequality \eqref{h sgrad} holds because $(g_k)$ is a fractional $s$-gradient of $u$, $(\tilde g_k)$ is a fractional $s$-gradient of $v$ and $g_{k}\le\tilde{g_{k}}$ almost everywhere. 

If $x\in E_\lambda$ and $y\in X\setminus E_\lambda$, then $(u-v)(y)=0$ and $h_k^\lambda(y)=0$ for all $k$.
Let $R=d(x,X\setminus E_\lambda)$ and let $l$ be such that $2^{-l-1}<R\le 2^{-l}$. Then $l\ge k$.
Using similar arguments as earlier in the proof and the properties of the functions $\ph_{i}$, the fact that $2B_i\subset B(x,R)$ with comparable radius for all $i\in I_x$, the doubling property, Lemma \ref{pointwise lemma} and Theorem \ref{median poincare}, we have 
\begin{align*}
|(u-v)(x)|
&\le\sum_{i\in I_x}|u(x)-m_u^\gamma(2B_i)|\\
&\le\sum_{i\in I_x}\Big(|u(x)-m_u^\gamma(B(x,R))|+|m_u^\gamma(B(x,R))-m_u^\gamma(2B_i)|\Big)\\
&\le C2^{-ls}\sum_{j=l-4}^l M_{\gamma/C}\,g_j(x)\\
&\le Cd(x,y)^s\tilde g_k(x).
\end{align*}
Hence $(h_k^\lambda)\in\D^s(u-v)$. 
Since $\mu(E_\lambda)\to 0$ as $\lambda\to\infty$, Lemma \ref{B norm abs cont} implies that $\|(h_k^\lambda)\|_{l^q(L^p(X))}\to0$ as  $\lambda\to0$.
\end{proof}
We conclude that $v\to u$ in $N^s_{p,q}(X)$.
\smallskip

\noindent {\sc Step 2:} General case.

\noindent Let $\eps>0$.
By the $5r$-covering theorem, there is a covering of $X$ by balls $B(a_j,1/2)$, $j\in\n$, such that balls $B(a_j,1/10)$ are disjoint and the balls $B(a_j,2)$ have bounded overlap. 
Let $B_j=B(a_j,1)$, $j\in\n$, and let $(\psi_j)$ be a partition of unity such that $\sum_{j=1}^{\infty}\psi_j=1$,
each $\psi_j$ is $L$-Lipschitz, $0\le\psi_j\le1$, and $\operatorname{supp}\psi_j\subset B_j$ for all $j\in\n.$

Let $u_j=u\psi_j$. Then
\begin{equation}\label{ux}
u(x)=\sum_{j=1}^{\infty}u_j(x),
\end{equation}
and the sum is finite for all $x\in X$. 
Lemma \ref{leibniz} shows that $u_j\in N^s_{p,q}(X)$ for each $j$ and $(g_{j,k})_{k\in\z}$, where
\[
g_{j,k}=
\begin{cases}
\big(g_k+2^{s k+2}|u|\big)\ch{B_j}\quad &\text{if }k<k_L,\\
\big(g_k+2^{k(s-1)}L|u|\big)\ch{B_j},\quad &\text{if }k\ge k_L,
\end{cases}
\] 
and $k_L$ is an integer such that $2^{k_L-1}<L\le 2^{k_L}$, is a fractional $s$-gradient of $u_j$.

Since $\operatorname{supp}u_j\subset B_j$, the first step of the proof
shows there are functions $v_j\in N^s_{p,q}(X)$ and open sets $\Omega_j\subset 2B_j$ such that
\renewcommand{\labelenumi}{(\roman{enumi})}
\begin{enumerate}
\item\label{holder 1j} $v_j=u_j$ in $X\setminus \Omega_j$, 
 $\operatorname{supp}v_j\subset 2B_j$, 
\item\label{holder 2j}$v_j\in N^s_{p,q}(X)$ is $\beta$-H\"older continuous,  
\item\label{holder 3j}$\|u_j-v_j\|_{N^s_{p,q}(X)}<2^{-j}\eps$,
\item\label{holder 4j}$\cH^{(s-\beta)p,q/p}_R(\Omega_j)<2^{-j/r}\eps$, where $r=\min\{1,q/p\}$.
\item\label{holder 5j}$(\tilde g_{j,k})_{k\in\z}$ is a fractional $s$-gradient of $v_j$.  
\end{enumerate}

We define $\Omega=\cup_{j=1}^{\infty}\Omega_j$, and show that function
$v=\sum_{j=1}^{\infty}v_j$ has properties \eqref{holder 1}-\eqref{holder 4} of Theorem \ref{holder representative B}. 
The first property follows from (i) and \eqref{ux}. 
The Netrusov-Hausdorff content estimate follows from (iv) and Lemma \ref{NH content subadd}. 
By \eqref{s holder},
$|v_j(x)-v_j(y)|\le C\lambda_j d(x,y)^{\beta}$ for all $x,y\in X$.
Since, by the proof above, the constant $\lambda_j$ depends on $\eps$ and on $j$, the H\"older continuity of the functions $v_j$ and the fact that $\operatorname{supp}v_j\subset 2B_j$ give H\"older continuity of $v$ only in bounded subsets of $X$.
By (iii), we have
\begin{equation}\label{norm of u-v}
\sum_{j=1}^{\infty}\|u_j-v_j\|_{N^s_{p,q}(X)}<\sum_{j=1}^{\infty}2^{-j}\eps=\eps,  
\end{equation}
that is, the series $\sum_{j=1}^{\infty}(u_j-v_j)$ convergences absolutely, and hence converges in the quasi-Banach space
$N^s_{p,q}(X)$. 
Since $u=\sum_{j=1}^{\infty}u_j$ is in $N^s_{p,q}(X)$, also $\sum_{j=1}^{\infty}v_j$ converges in $N^s_{p,q}(X)$. 
Moreover, by \eqref{norm of u-v} and \eqref{quasi norm ie}, we obtain 
\[
\|u-v\|_{N^s_{p,q}(X)}^r
\le C\sum_{j=1}^{\infty}\|u_j-v_j\|_{N^s_{p,q}(X)}^r<c\eps^r.
\]
\end{proof}

\begin{proof}[Proof of Theorem \ref{holder representative TL}]
The proof for a Triebel--Lizorkin function $u\in M^s_{p,q}(X)$ requires only small modifications. 
To obtain the desired Hausdorff content estimate, a counterpart of Claim 2 of the Besov case, 
let $(g_k)_{k\in\z}\in \D^s(u)\cap L^p(X;l^q)$. Then $g=\sup_{k\in\z} g_k$ belongs to $L^p(X)$ and is an $s$-gradient of $u$.
It follows that
\[\begin{split}
E_\lambda\subset
\bigg\{x\in X: C\sup_{i\ge-6}2^{-i(s-\beta)}\Big(\,\vint{B(x,2^{-i})}g^p\,d\mu\Big)^{1/p} >\lambda\bigg\}=F_\lambda.
\end{split}
\]
Hence, using a standard argument, we obtain that $\cH^{(s-\beta)p}_{2^6}(F_\lambda)\le \lambda^{-p}\|g\|_{L^p(X)}^p$.

Moreover, Lemma \ref{IntForGradOfExtensionBesov} is replaced by the estimate
\begin{equation}\label{grad est}
\|(\tilde g_k)\|_{L^p(X;l^q)}\le C\|(g_k)\|_{L^p(X;l^q)}.
\end{equation}
Since
\[
\sum_{k\in\z}\tilde g_k^q\le \sum_{k\in\z}\sum_{j\in\z}2^{-|j-k|\delta q}(M_{\gamma/C}\,g_j)^q\le C \sum_{j\in\z}(M_{\gamma/C}\,g_j)^q,
\]
when $q<\infty$, and $\sup_{k\in\z}\tilde g_k\le C\sup_{k\in\z}M_{\gamma/C}\,g_k$, estimate
\eqref{grad est} follows from \eqref{Lplq boundedness of median maximal op}.

Finally, when $s=1$, we use Leibniz rule \cite[Lemma 5.20]{HjKi} instead of Lemma \ref{leibniz}.
\end{proof}

\section{Appendix}\label{app}
The fact that Haj\l asz--Besov and Haj\l asz--Triebel--Lizorkin spaces are complete (Banach spaces when $p,q\ge1$ and quasi-Banach spaces otherwise) has not been proved in earlier papers.

\begin{theorem}\label{m1p banach}
The spaces $N^s_{p,q}(X)$ and $M^s_{p,q}(X)$ are complete for all $0<s<\infty$, $0<p,q\le\infty$. 
\end{theorem}
 
\begin{proof}
We prove the Besov case, the proof for Triebel--Lizorkin spaces is similar.
Let $(u_i)_i$ be a Cauchy sequence in $N^s_{p,q}(X)$. 
Since $L^p(X)$ is complete, there exists a function $u\in L^p(X)$ such that $u_i\to u$ in $L^p(X)$ as $i\to\infty$. 
We will show that $(u_i)_i$ converges to $u$ in $N^s_{p,q}(X)$.

We may assume (by taking a subsequence) that 
\[
\|u_{i+1}-u_i\|_{N^s_{p,q}(X)}\le 2^{-i}
\] 
for all $i\in\n$ and that $u_i(x)\to u(x)$ as $i\to\infty$ for almost all $x\in X$. 
Hence, for each $i\in\n$, there exists $(g_{i,k})_k\in l^q(L^p(X))$ and a set $E_i$ of zero measure such that 
\[
|(u_{i+1}-u_i)(x)-(u_{i+1}-u_i)(y)|
\le d(x,y)^s(g_{i,k}(x)+g_{i,k}(y))
\]
for all $x,y\in X\setminus E_i$ satisfying $2^{-k-1}\le d(x,y)<2^{-k}$, and that $\|(g_{i,k})\|_{l^q(L^p(X))}\le 2^{-i}$.
This implies that
\[
|(u_{i+k}-u_i)(x)-(u_{i+k}-u_i)(y)|
\le d(x,y)^s\Big( \sum_{j=i}^\infty g_{j,k}(x)+ \sum_{j=i}^\infty g_{j,k}(y)\Big)
\]
for all $i,k\ge1$ for almost all $x,y\in X$. 
This together with the pointwise convergence shows that, letting $k\to\infty$, we have
\[
|(u-u_i)(x)-(u-u_i)(y)|
\le d(x,y)^s\Big( \sum_{j=i}^\infty g_{j,k}(x)+ \sum_{j=i}^\infty g_{j,k}(y)\Big).
\]
Hence $u-u_i$ has a fractional $s$-gradient $(\sum_{j=i}^\infty g_{j,k})_k$. 

When $p,q\ge 1$, we have $\|(\sum_{j=i}^\infty g_{j,k})_k\|_{l^q(L^p(X))}\le 2^{-i+1}$. 
If $0<\min\{p,q\}<1$, then, by \eqref{quasi norm ie}, there is $0<r<1$ such that
\[
\big\|(\sum_{j=i}^\infty g_{j,k})_k\big\|_{l^q(L^p(X))}^r
\le C\sum_{j=i}^\infty\|(g_{j,k})_k\|_{l^q(L^p(X))}^r
\le C2^{-ri}.
\] 
Hence, in both cases, $u-u_i\in N^s_{p,q}(X)$ and $u_i\to u$ in $N^s_{p,q}(X)$. Thus $u\in N^s_{p,q}(X)$ and the claim follows.
\end{proof}

\noindent {\bf Acknowledgements:} The research was supported by the Academy of Finland, grants no.\ 135561 and 272886. 
Part of this research was conducted during the visit of the second author to Forschungsinstitut f\"ur Mathematik of ETH Z\"urich, and she wishes to thank the institute for the kind hospitality.

\vspace{0.5cm}
\noindent
\small{\textsc{T.H.},}
\small{\textsc{Department of Mathematics and Statistics},}
\small{\textsc{P.O. Box 35},}
\small{\textsc{FI-40014 University of Jyv\"askyl\"a},}
\small{\textsc{Finland}}\\
\footnotesize{\texttt{toni.heikkinen@aalto.fi}}

\vspace{0.3cm}
\noindent
\small{\textsc{H.T.},}
\small{\textsc{Department of Mathematics and Statistics},}
\small{\textsc{P.O. Box 35},}
\small{\textsc{FI-40014 University of Jyv\"askyl\"a},}
\small{\textsc{Finland}}\\
\footnotesize{\texttt{heli.m.tuominen@jyu.fi}}


\begin{thebibliography}{000}

\bibitem{A2}D. R. Adams:
Besov capacity redux,
Problems in mathematical analysis. No. 42. J. Math. Sci. (N. Y.) 162 (2009), no. 3, 307--318.

\bibitem{Am}
L. Ambrosio:
Fine properties of sets of finite perimeter in doubling metric measure spaces,
Calculus of variations, nonsmooth analysis and related topics,
Set-Valued Anal. 10 (2002), no. 2-3, 111--128.

\bibitem{Ao}
T. Aoki:
Locally bounded linear topological spaces,
Proc. Imp. Acad. Tokyo 18, (1942). 588--594. 

\bibitem{BHS}
B. Bojarski, P. Haj\l asz, P. Strzelecki:
Improved $C^{k,\lambda}$ approximation of higher order Sobolev functions in norm and capacity, 
Indiana Univ. Math. J. 51 (2002), no. 3, 507--540. 

\bibitem{CW}
R. R Coifman and G. Weiss: 
Analyse harmonique non-commutative sur certains espaces homog\`enes,
Lecture Notes in Mathematics, Vol.242. Springer-Verlag, Berlin-New York, 1971. 

\bibitem{FS}
C. Fefferman and E. M. Stein:
Some maximal inequalities,
Amer. J. Math. 93 (1971), 107--115.

\bibitem{Fu}
N. Fujii:
A condition for a two-weight norm inequality for singular integral operators,
Studia Math. 98 (1991), no. 3, 175--190.


\bibitem{GKS}
A. Gogatishvili, P. Koskela and N. Shanmugalingam:
Interpolation properties of Besov spaces defined on metric spaces,
Math. Nachr. {\bf 283} (2010), no. 2, 215--231.

\bibitem{GKZ} 
A. Gogatishvili, P. Koskela and Y. Zhou:
Characterizations of Besov and Triebel--Lizorkin Spaces on Metric Measure Spaces, 
Forum Math. 25 (2013), no. 4, 787--819.

\bibitem{GLY}
L. Grafakos, L. Liu and D. Yang:
Vector-valued singular integrals and maximal functions on spaces of homogeneous type, 
Math. Scand. 104 (2009), 296--310.

\bibitem{H} 
P. Haj\l asz:
Sobolev spaces on an arbitrary metric space, 
Potential Anal. 5 (1996), 403--415.

\bibitem{HjKi}
P. Haj\l asz and J. Kinnunen: 
H\"older quasicontinuity of Sobolev functions on metric spaces,
Rev. Mat. Iberoamericana 14 (1998), no.3, 601--622.

\bibitem{HMY}
Y. Han, D. M\"uller, and D. Yang:
A theory of Besov and Triebel--Lizorkin spaces on metric measure spaces modeled on Carnot--Carath\'eodory spaces,
Abstr. Appl. Anal. 2008, Art. ID 893409, 250 pp.

\bibitem{HN}
L. I. Hedberg and Y. Netrusov:
An axiomatic approach to function spaces, spectral synthesis, and Luzin approximation, 
Mem. Amer. Math. Soc. 188 (2007), no. 882.

\bibitem{HIT}
T. Heikkinen, L. Ihnatsyeva and H. Tuominen:
Measure density and extension of Besov and Triebel--Lizorkin functions,
preprint 2014, http://arxiv.org/abs/1409.0379

\bibitem{HKT}
T. Heikkinen, P. Koskela and H. Tuominen:
Generalized Lebesgue points, in preparation.

\bibitem{HeTu}
T. Heikkinen and H. Tuominen:
 Smoothing properties of the discrete fractional maximal operator on Besov and Triebel--Lizorkin spaces, 
 Publ. Mat. 58 (2014), no. 2. 379--399.
 


\bibitem{JPW}
B. Jawerth, C. Perez and G. Welland: 
The positive cone in Triebel--Lizorkin spaces and the relation among potential and maximal operators,
Harmonic analysis and partial differential equations (Boca Raton, FL, 1988), 71--91, 
Contemp. Math., 107, Amer. Math. Soc., Providence, RI, 1990. 

\bibitem{JT}
B. Jawerth and A. Torchinsky: 
Local sharp maximal functions,
J. Approx. Theory 43 (1985), no. 3, 231--270. 

\bibitem{KiTu}
J. Kinnunen and H. Tuominen:
Pointwise behaviour of $M^{1,1}$ Sobolev functions,
Math. Z. 257 (2007), no. 3, 613--630.

\bibitem{KS}
P. Koskela and E. Saksman:
Pointwise characterizations of Hardy--Sobolev functions,
Math. Res. Lett. 15 (2008), no. 4, 727--744.



\bibitem{KYZ}
P. Koskela, D. Yang and Y. Zhou:
Pointwise Characterizations of Besov and Triebel--Lizorkin Spaces and Quasiconformal Mappings,
Adv. Math. 226 (2011), no. 4, 3579--3621.

\bibitem{KP}
V. G. Krotov and M.A. Prokhorovich:
The Luzin approximation of functions from the classes $W^p_\alpha$ on metric spaces with measure, 
(Russian) Izv. Vyssh. Uchebn. Zaved. Mat. 2008, no. 5, 55--66; translation in Russian Math. (Iz. VUZ) 52 (2008), no. 5, 47--57.


\bibitem{Le}
A. K. Lerner, 
A pointwise estimate for the local sharp maximal function with applications to singular integrals, 
Bull. London Math. Soc. 42 (2010), 843--856.

\bibitem{LP}
A. K. Lerner and C. P\'erez: 
Self-improving properties of generalized Poincar\'e type inequalities throught rearrangements, 
Math. Scand. 97 (2) (2005), 217--234.




\bibitem{L}
F. C. Liu:
A Luzin type property of Sobolev functions,
Indiana Univ. Math. J. 26 (1977), no. 4, 645--651. 

\bibitem{MS}
R. A. Mac\'ias and C. Segovia: 
A decomposition into atoms of distributions on spaces of homogeneous type,
Adv. in Math.  33 (1979), no.3, 271--309.  

\bibitem{Ma}
J. Mal\'y:
H\"older type quasicontinuity,
Potential Anal. 2 (1993), no. 3, 249--254. 

\bibitem{MZ}
J. H. Michael and W. P. Ziemer:
Lusin type approximation of Sobolev functions by smooth functions,
Contemp. Math. 42 (1985), 135--167. 

\bibitem{MY}
D. M\"uller and D. Yang:
A difference characterization of Besov and Triebel--Lizorkin spaces on RD-spaces,
Forum Math. {\bf 21} (2009), no. 2, 259--298.

\bibitem{Ne92}
Y. V. Netrusov:
Metric estimates for the capacities of sets in Besov spaces. (Russian), 
Trudy Mat. Inst. Steklov. 190 (1989), 159--185. 
translation in Proc. Steklov Inst. Math. 1992, no. 1, 167--192. 


\bibitem{Ne96}
Y. V. Netrusov:
Estimates of capacities associated with Besov spaces. (Russian), 
Zap. Nauchn. Sem. S.-Peterburg. Otdel. Mat. Inst. Steklov. (POMI) 201 (1992), 124--156, 
translation in J. Math. Sci. 78 (1996), no. 2, 199--217 

\bibitem{PT}
J. Poelhuis and A. Torchinsky:
Medians, continuity, and vanishing oscillation, 
Studia Math. 213 (2012), no. 3, 227--242.

\bibitem{Ro}
S. Rolewicz:
On a certain class of linear metric spaces,
Bull. Acad. Polon. Sci. Cl. III. 5 (1957), 471--473.

\bibitem{Sa}
Y. Sawano:
Sharp estimates of the modified Hardy--Littlewood maximal operator on the nonhomogeneous space via covering lemmas,
Hokkaido Math. J. 34 (2005), no. 2, 435--458.



\bibitem{SYY}
N. Shanmugalingam, D. Yang and W. Yuan:
Newton--Besov Spaces and Newton--Triebel--Lizorkin Spaces on Metric Measure Spaces,
to appear in Positivity, 
http://dx.doi.org/10.1007/s11117-014-0291-7

\bibitem{Sto}
B. M. Stocke:
A Lusin type approximation of Bessel potentials and Besov functions by smooth functions,
Math. Scand. 77 (1995), no. 1, 60--70. 

\bibitem{St}
J-O. Str\"omberg:
Bounded mean oscillation with Orlicz norms and duality of Hardy spaces, 
Indiana Univ. Math. J. 28 (1979), no. 3, 511--544.

\bibitem{Sw1}
D.  Swanson:
Pointwise inequalities and approximation in fractional Sobolev spaces,
Studia Math. 149 (2002), no. 2, 147--174. 

\bibitem{Sw2}
D. Swanson:
Approximation by H\"older continuous functions in a Sobolev space,
Rocky Mountain J. Math. 44 (2014), no. 3, 1027--1035. 

\bibitem{Tr} 
H. Triebel:
Theory of Function Spaces, 
Birkh\"auser Verlag, Basel, 1983.

\bibitem{Tr2}
 H. Triebel:
Theory of function spaces. II,
Monographs in Mathematics, 84. Birkh\"auser Verlag, Basel, 1992.

\bibitem{Y}
D. Yang:
New characterizations of Haj\l asz--Sobolev spaces on metric spaces,
Sci. China Ser. A {\bf 46} (2003), 675--689.

\bibitem{YZ}
D. Yang and Y. Zhou:
New properties of Besov and Triebel--Lizorkin spaces on RD-spaces,
Manuscripta Math. {\bf 134} (2011), no. 1-2, 59--90.

\bibitem{Zh}
Y. Zhou: 
Fractional Sobolev extension and imbedding,
Trans. Amer. Math. Soc. 367 (2015), no. 2, 959--979

\end{thebibliography}
\end{document}